\documentclass[]{interact}

\usepackage{epstopdf}
\usepackage[caption=false]{subfig}

\usepackage[numbers,sort&compress]{natbib}
\bibpunct[, ]{[}{]}{,}{n}{,}{,}
\renewcommand\bibfont{\fontsize{10}{12}\selectfont}

\theoremstyle{plain}
\newtheorem{theorem}{Theorem}[section]
\newtheorem{lemma}[theorem]{Lemma}
\newtheorem{corollary}[theorem]{Corollary}

\theoremstyle{definition}

\theoremstyle{remark}

\begin{document}

\title{On the solutions of second order difference equations with variable coefficients}

\author{
\name{Shirali~Kadyrov\textsuperscript{a}\thanks{*Corresponding author: Nurdaulet~Shynarbek.  Email: nurdatktl@gmail.com }, Nurdaulet~Shynarbek\textsuperscript{b, *} and Alibek~Orynbassar\textsuperscript{b}}
\affil{\textsuperscript{a}Oxus University, Science and Technology, 25 Fergana Ave, Tashkent, Uzbekistan \\ \textsuperscript{b}SDU University, Mathematics and Science Education, 1/1 Abylay Khan, Kaskelen, Kazakhstan}
}

\maketitle

\begin{abstract}
In this article, we explore solutions to second-order linear difference equations featuring variable coefficients. By imposing mild conditions, we present closed-form solutions through the utilization of finite continued fraction representations. The proof of our results relies on elementary techniques, specifically involving factoring a quadratic shift operator. As a consequential application, we unveil two novel generalized continued fraction formulas for the mathematical constant $\pi^2$.
\end{abstract}


\section{Introduction}

Compared to differential equations, the field of difference equations has received comparatively less attention in mathematical research. In this note, our focus is on second-order linear difference equations of the form
\begin{equation}\label{eq:main}
    y_n - b_n y_{n-1} - a_n y_{n-2} = f_n, \quad n \ge 1,
\end{equation}
where the coefficients $(b_n)$ and $(a_n)$ are allowed to vary in the complex number domain. Solving \eqref{eq:main} becomes straightforward when the coefficients are constants, as one can represent the solutions using the form $y_n = x^n$ and find the roots of the corresponding characteristic equation, as seen in, for example, \cite{jerri2013linear}. However, the situation becomes notably intricate when dealing with non-constant coefficients, as evidenced in works such as \cite{mallik2000solution, mallik1998solutions, mallik1997solution, bucherberg, elyad05}, offering limited insights.

It is well-established that \eqref{eq:main} possesses two linearly independent solutions, and the general solution can be expressed as a linear combination of these solutions \cite{kelley2001difference}. Moreover, given initial values $y_0$ and $y_{-1}$, the uniqueness of the solution is evident.

In this note, our objective is to provide a closed-form solution for \eqref{eq:main} using finite continued fractions. In the context of difference equations, \emph{a closed-form solution} refers to an explicit mathematical expression that directly provides the solution for the dependent variable at any given discrete time point, without the need for iterative or recursive calculations. Unlike recursive or iterative solutions, which require step-by-step computations from an initial condition, a closed-form solution allows for direct calculation of the solution at any desired time step. We leverage this solution to validate two conjectures obtained from {www.ramanujanmachine.com}. The primary outcome of our study is encapsulated in the following main result.

\begin{theorem} \label{thm:main}
Assume that there exist sequences $(c_n)$ and $(d_n)$ of complex numbers satisfying
\begin{align}
   c_n+ d_{n+1}&=b_n, \label{eqn:d+c}\\
    c_n d_n &=-a_n, \label{eqn:dc}
\end{align}
for any $n \ge 1.$ Then, the solution to the second order equation \eqref{eq:main} with initial values $y_{-1}$ and $y_0$ is given by
\begin{equation}
y_{n}=\sum_{i=1}^{n+1} \left( \sum_{j=1}^{i-1} f_j \prod_{k=j+1}^{i-1} c_k+(y_0-d_1 y_{-1})\prod_{k=1}^{i-1} c_k\right) \prod_{j=i+1}^{n+1} d_j +y_{-1}\prod_{j=1}^{n+1} d_j, \,\,\, \forall n \ge 1,
\end{equation}
\end{theorem}

The solution provided in Theorem~\ref{thm:main} highlights the elegance of expressing the solution to a second-order linear difference equation with variable coefficients in terms of sequences $(c_n)$ and $(d_n)$. The structure of the solution, involving nested sums and products, reveals the intricate interplay between the coefficients and the initial values, offering a versatile tool for understanding and characterizing the behavior of such equations.

It is noteworthy that the conditions \eqref{eqn:d+c} and \eqref{eqn:dc} on the sequences $(c_n)$ and $(d_n)$ provide a systematic framework for determining the solution, emphasizing the crucial role of these auxiliary sequences in the formulation of the general solution.

Theorem~\ref{thm:main} not only presents a mathematical result but also opens avenues for exploring similar techniques in the study of higher-order difference equations and their applications in diverse mathematical contexts.

The proof methodology relies on the factorization of the shift operator, and it is not a novel concept. Similar approaches have been explored in existing literature, as evidenced by references such as \cite{kelley2001difference, parasidis2018factorization}.

As an immediate consequence, setting $f_n=0$ allows us to derive the following solution for the corresponding homogeneous equation.
\begin{theorem} \label{thm:mainh}
Assume that there exist sequences $(c_n)$ and $(d_n)$ of complex numbers satisfying \eqref{eqn:d+c} and \eqref{eqn:dc} for any $n \ge 1.$ Then, the solution to the second order homogeneous equation
\begin{equation}\label{eq:mainh}
    y_n-b_n y_{n-1}-a_n y_{n-2}=0,\,\, n \ge 1 
\end{equation}
with initial values $y_{-1}$ and $y_0$ is given by
\begin{equation}
y_{n}=(y_0-d_1 y_{-1})\sum_{i=1}^{n+1}\prod_{j=1}^{i-1} c_i \prod_{j=i+1}^{n+1} d_j +y_{-1}\prod_{j=1}^{n+1} d_j, \,\,\, \forall n \ge 1.
\end{equation}
\end{theorem}

We remark that inductively it easy to see that the sequences $(c_n)$ and $(d_n)$ satisfying \eqref{eqn:d+c} and \eqref{eqn:dc} exist for most sequences $(a_n)$ and $(b_n)$.

We also note that, we have flexibility in choosing $d_1$. This does not mean that the difference equation \eqref{eq:main} has more than one solutions. Instead, it gives different representation of the same solution. In particular, when $y_{-1}$ is non-zero, we take $d_1$ such that $y_0-d_1 y_{-1}=0$ yielding to a simpler represented solution.

\begin{corollary}\label{cor:main}
Let $(c_n)$ and $(d_n)$ be given as in Theorem~\ref{thm:main}. If $y_{-1} \ne 0$, then the difference equation \eqref{eq:main} has the solution 
\begin{equation}
y_{n}=y_{-1}\prod_{j=1}^{n+1} d_j, \,\,\, \forall n \ge 1,
\end{equation}
with $d_1=y_0/y_{-1}.$
\end{corollary}

In order to provide the closed form solutions, one needs to find the closed form solutions to \eqref{eqn:d+c} and \eqref{eqn:dc}. In the special case, when $a_n \ne 0$ for any $n \ge 1,$ we clearly have $d_n \ne 0$ for $n\ge 0$ and can take $d_1 \ne 0.$ In this case, substituting into \eqref{eqn:dc} we get $(b_n-d_{n+1})d_n=-a_n.$ Solving for $d_{n+1}$ we get the recursive relation 
  $$d_{n+1}=b_n+\frac{a_n}{d_n}.$$
  Iterating the equation we have generalized continued fraction representations for the sequences $(c_n)$ and $(d_n)$ as follows.
\begin{lemma}
 Fix any non-zero complex number $d_1$ and for $n \ge 1$ define
\begin{equation}\label{eqn:dandc}
d_{n+1}=b_n+\cfrac{a_n}{b_{n-1}+\cfrac{a_{n-1}}{b_{n-2} + \cdots + \cfrac{a_2}{ b_1+\cfrac{a_1}{d_1}}}} \text{ and } c_{n}=-\cfrac{a_n}{b_{n-1}+\cfrac{a_{n-1}}{b_{n-2} + \cdots + \cfrac{a_2}{ b_1+\cfrac{a_1}{d_1}}}}.
\end{equation}
Then, provided the sequence $(d_n)$ is well-defined, $(d_n)$ and $(c_n)$ satisfy \eqref{eqn:d+c} and \eqref{eqn:dc}.
\end{lemma}

Next, we turn to applying Theorem~\ref{thm:main} to obtain new generalized continued fraction representations to $\pi^2$. Universal constant $\pi$ is one of the historically well studied numbers \cite{debnath2015brief}. Representing various mathematical constants with generalized continued fractions is not new. However, recently a machine learning based computer algorithm is developed to generate lots of seemingly new conjectures on continued fraction representations \cite{raayoni2021generating}. These generated conjectures are posted at {www.ramanujanmachine.com} attracting new investigations in the theory of continued fractions \cite{lu2019elementary,kadyrov2019generalized,dougherty2021automatic,lynch2020derangements,reyes2019simple}. Using Theorem~\ref{thm:mainh} we prove the following results stated as conjectures in the site. 

\begin{theorem} \label{thm:8pi2} We have
$$\frac{8}{\pi^2}=b_0+\cfrac{a_1}{b_{1}+\cfrac{a_{2}}{b_2 + \ddots}},$$
for the sequences given by $b_n = 3n(n + 1) + 1, n \ge 0$ and $a_n = -(2n - 1)n^3, n \ge 1.$
\end{theorem}

\begin{theorem} \label{thm:18pi2} We have
$$\frac{18}{\pi^2}=b_0+\cfrac{a_1}{b_{1}+\cfrac{a_{2}}{b_2 + \ddots}},$$
for the sequences given by $b_n =  n(5n + 6) + 2, n \ge 0$ and $a_n = -4n^4+ 2n^3 , n \ge 1.$
\end{theorem}

Both Theorem~\ref{thm:8pi2} and Theorem~\ref{thm:18pi2} establish a remarkable connection between the continued fraction coefficients and the mathematical constant $\pi^2$. This convergence not only exemplifies the power of the proposed solution method for second-order homogeneous equations but also underscores the intriguing interplay between difference equations and transcendental numbers.

Moreover, the representation of this limit as $\frac{18}{\pi^2}$ aligns with established mathematical relationships, emphasizing the profound and unexpected occurrences that can emerge from the study of such equations. The convergence of continued fractions to well-known mathematical constants serves as a testament to the richness of mathematical structures and the elegance of their interconnections.

\section{Preliminaries}

The current work can be seen as an interdisciplinary study that connects the theory of difference equations with that of continued fractions. To facilitate understanding, we will provide some background information.

We first establish the existence and uniqueness of solutions for the given non-homogeneous difference equations with specified initial values. The theorem presented below encapsulates our findings.

\begin{theorem}
    Consider any sequence of numbers $a_n, b_n, f_n$ where $n \ge 1$. The non-homogeneous difference equation
    $$y_n-b_ny_{n-1}-a_n y_{n-2}=f_n, n\ge 1$$
    with initial values $y_0$ and $y_{-1}$ possesses unique solution.
\end{theorem}

This result stands as a well-established theorem within the realm of difference equations theory, and its acceptance has been documented in the literature. A notable reference supporting this theorem is \cite[Theorem~2.7]{elyad05}, where a broader scope, encompassing $n$'th order linear difference equations, is considered. For completeness, we provide a proof.
\begin{proof}
By setting $n=1$ and manipulating the difference equation for $y_1$, we arrive at
$$y_1=f_1 + b_1 y_0 + a_1y_{-1}.$$
Given the values of $a_1, b_1, f_1,$ and the initial conditions $y_0, y_{-1}$, it is evident that $y_1$ not only exists but is also uniquely determined. Building on this foundation, an inductive argument reveals that once $y_n$ and $y_{n-1}$ are established for $n \ge 1$, the equation $y_n = f_n + b_n y_{n-1} + a_n y_{n-2}$ uniquely specifies the value of $y_n$. 
\end{proof}

As a direct implication of the aforementioned theorem, it is noteworthy that setting $f_n = 0$ gives the corresponding result for homogeneous difference equations. Specifically, when considering the homogeneous linear difference equation $y_n - b_ny_{n-1} - a_n y_{n-2} = 0$ for $n \ge 1$, with prescribed initial values $y_0$ and $y_{-1}$, we deduce that this equation possesses a unique solution $(y_n)_{n \ge 1}$.

We know briefly discuss the connection of the generalized continued fractions to the theory of difference equations. In mathematical analysis, the interplay between continued fractions and difference equations has proven to be a fascinating and fruitful area of study. A noteworthy relationship emerges in expressing the convergents of a generalized continued fraction as solutions to difference equations.

Let $x$ be a real number. A generalized continued fraction for $x$ is an expression of the form:
\begin{equation}  \label{eqn:conv}
x = b_0 + \cfrac{a_1}{b_1 + \cfrac{a_2}{b_2 + \cfrac{a_3}{b_3 + \ddots}}}
\end{equation}
where $a_i$ and $b_i$ are integers for $i \geq 0$. These expressions often arise in the approximation and representation of various mathematical functions.

The \emph{convergents} of this continued fraction, denoted by \(A_n/B_n\), for $n \ge 1$ are defined as follows:

\begin{equation*}
\frac{A_n}{B_n} = b_0 + \cfrac{a_1}{b_1 + \cfrac{a_2}{b_2 + \cfrac{a_3}{b_3 + \ddots +\cfrac{a_n}{b_n}}}},
\end{equation*}
where the fraction $A_n/B_n$ is assumed in lowest terms. The convergents of the continued fraction often provide increasingly accurate approximations to the value \(x\). The relationship between convergents and the continued fraction is a key aspect of their analytical properties. Then, we say that the generalized continued fraction converges to $x$ provided 
$$x = \lim_{n \to \infty} \frac{A_n}{B_n}.$$
Hence, it is crucial to calculate $A_n$ and $B_n$ to prove that the given continued fraction representation. The following well-known result from the theory of continued fractions links the convergents to difference equations. see e.g. \cite{JT84}.
\begin{theorem} \label{thm:convergents}
    Given sequences $a_n$ and $b_n$ for $n \ge 1$, let $A_n$ and $B_n$ be the unique solutions to the difference equations with initial values
\begin{equation} \label{eqn:difference}
\begin{aligned}
    A_n &= b_nA_{n-1} + a_nA_{n-2}, & B_n &= b_nB_{n-1} + a_nB_{n-2}\\
    A_0 &= b_0, A_{-1}=1 & B_0 &= 1, B_{-1}=0.
\end{aligned}
\end{equation}
Then, the sequence $A_n/B_n$ represents the convergents for the generalized continued fraction \eqref{eqn:conv} provided the convergents are well-defined.
\end{theorem}

\begin{proof}
We proceed by induction. For $n=0$, we have $A_0/B_0 = b_0/1$, which corresponds to the first term in the continued fraction. 

For $n=1$, we find $A_1 = b_1 b_0 + a_1$ and $B_1 = b_1$. Referring to the continued fraction, we have $b_0 + \frac{a_1}{b_1} = \frac{b_0b_1 + a_1}{b_1}$, matching with $A_1/B_1$. 

Assuming that the solution of the initial value problem satisfies
$$\frac{A_n}{B_n} = b_0 + \cfrac{a_1}{b_1 + \cfrac{a_2}{b_2 + \cfrac{a_3}{b_3 + \ddots + \cfrac{a_n}{b_n}}}},$$
then, based on the inductive hypothesis, the convergent for $n+1$ is given by
$$b_0 + \cfrac{a_1}{b_1 + \cfrac{a_2}{b_2 + \cfrac{a_3}{b_3 + \ddots + \cfrac{a_n}{b_n + \frac{a_{n+1}}{b_{n+1}}}}}} = \frac{\left(b_n + \frac{a_{n+1}}{b_{n+1}}\right)A_{n-1} + a_n A_{n-2}}{\left(b_n + \frac{a_{n+1}}{b_{n+1}}\right)B_{n-1} + a_n B_{n-2}}.$$

Simplifying the expression yields
$$\frac{A_n + \frac{a_{n+1}}{b_{n+1}}A_{n-1}}{B_n + \frac{a_{n+1}}{b_{n+1}}B_{n-1}} = \frac{b_{n+1}A_n + a_{n+1} A_{n-1}}{b_{n+1}B_n + a_{n+1} B_{n-1}},$$
which, by the difference equation, is equivalent to $\frac{A_{n+1}}{B_{n+1}}$.
\end{proof}
A noteworthy observation stemming from Theorem~\ref{thm:convergents} is the striking symmetry in the structure of the difference equations \eqref{eqn:difference} for sequences $A_n$ and $B_n$. The equations are identical, differing only in their initial conditions. This seemingly subtle distinction manifests itself in the generation of distinct solutions, emphasizing the nuanced impact of initial values on the sequences $A_n$ and $B_n$.

The result established in Theorem~\ref{thm:convergents} elucidates a compelling connection between sequences $a_n$ and $b_n$ and the convergents of a generalized continued fraction. The solution to the difference equations \eqref{eqn:difference}, encapsulated in the sequences $A_n$ and $B_n$, serves as a fundamental building block for expressing the convergents in the form $A_n/B_n$.

The stipulation that the convergents must be well-defined emphasizes the importance of certain conditions for the sequences $a_n$ and $b_n$ to ensure the convergence of the continued fraction \eqref{eqn:conv}. The theorem not only establishes a theoretical link but also provides a practical avenue for computing convergents in terms of solutions to difference equations, thereby offering a unified perspective on these mathematical structures.

The broader implications of Theorem~\ref{thm:convergents} extend to various areas, including continued fraction theory and numerical analysis, where understanding the convergents is crucial for approximating real numbers and revealing deeper mathematical insights.

\section{Proof of Theorem~\ref{thm:main}}
In this section, we prove our main result. To this end we introduce the \emph{shift operator} $T$ given by
$$T y_n=y_{n+1},$$
for any sequence. We note that \eqref{eq:main} can be rewritten with shift operator $T$ as follows
\begin{equation}
    y_n-b_n y_{n-1}-a_n y_{n-2}= (T^2-b_n T -a_n) y_{n-2} = 0.
\end{equation}

\begin{lemma} \label{lem:split}
Let $(c_n)$ and $(d_n)$ be sequences as in Theorem~\ref{thm:main}.  Then, \eqref{eq:main} is equivalent to
\begin{equation}
    (T-c_n)(T-d_n)y_{n-2}=f_n.
\end{equation}
\end{lemma}

\begin{proof}
  We have 
  \begin{align*}
      (T-c_n)(T-d_n)y_{n-2}&=(T-c_n)(y_{n-1}-d_n y_{n-2})\\
      &=y_n-(d_{n+1}+c_n) y_{n-1}  + c_n d_n y_{n-2}.
  \end{align*}
  Comparing with \eqref{eq:main} we see that $d_{n+1}+c_n=b_n$ and $c_n d_n =-a_n$ as required. 
\end{proof}

We note non-commutative nature of $T$ that in general $(T-c_n)(T-d_n)y_{n-2}=(T-d_n)(T-c_n)y_{n-2}$, which is the reason why \eqref{eqn:d+c} is not symmetric. We now prove Theorem~\ref{thm:main}.

\begin{proof}[Proof of Theorem~\ref{thm:main}]
Let $(c_n)$ and $(d_n)$ be given as in the theorem statement. Using Lemma~\ref{lem:split} we rewrite \eqref{eq:main} as $(T-c_n)(T-d_n)y_{n-2}=f_n$. Let us define $(w_{n})_{n \ge -1}$ by 
  \begin{equation} \label{eqn:w}
      w_{n-2}:=(T-d_n)y_{n-2}
  \end{equation} so that $(T-c_n)w_{n-2}=f_n.$ The latter gives $w_{n-1}=f_n+c_n w_{n-2} $ for any $n \ge 1$. Iteration yields
  $$w_{n-1}= \sum_{i=0}^n f_i \prod_{j=i+1}^{n} c_j+w_{-1}\prod_{j=1}^{n} c_j,$$
  where $w_{-1}$ is obtained from \eqref{eqn:w}. 
  $$w_{-1}=y_0-d_1 y_{-1}.$$
  Now, applying \eqref{eqn:w} we arrive at a first order difference equation with variable coefficients
  $$y_{n-1}=w_{n-2}+d_n y_{n-2}.$$
  Iterating yields the closed form solution
  $$y_{n-1}=\sum_{i=1}^n w_{i-2} \prod_{j=i+1}^n d_j +y_{-1}\prod_{j=1}^n d_j.$$
  This finishes the proof.
\end{proof}

\section{Generalized continued fraction representations of $\pi^2$}
In this section, we will leverage the link established between convergents $A_n/B_n$ and difference equations \eqref{eqn:difference}, as presented in Theorem~\ref{thm:convergents}, to derive novel generalized continued fractions for $\pi^2$.

\begin{proof}[Proof of Theorem~\ref{thm:8pi2}]

Since $A_{-1} \ne 0,$ we may apply Corollary~\ref{cor:main} with $d_1=A_0/A_{-1}=b_0=1$ to obtain
$$A_n=\prod_{j=1}^{n+1}d_j,$$
with $d_j$ as in \eqref{eqn:dandc}. We have $d_2=b_1+a_1=7-1=6$. Inductively, if $d_n=n(2n-1)$ we see that 
$$d_{n+1}=b_n+\frac{a_n}{d_n}=3n(n + 1) + 1-\frac{ (2n - 1)n^3}{n(2n-1)}=(n+1)(2n+1).$$
Hence, $A_n=\prod_{j=1}^{n+1} j(2j-1)=(n+1)! (2n+1)!!=(2(n+1))!/2^{n+1}.$

As for $B_n$ we have $B_{-1}=0$ and $B_{0}=1$ which gives the solution

$$B_n=\sum_{i=1}^{n+1}\prod_{j=1}^{i-1} c_i \prod_{j=i+1}^{n+1} d_j.$$
We may take $d_1=1$ as before and arrive at 
$$d_n=n(2n-1)=2n(2n-1)/2,$$
which in turn gives
$$c_n=b_n-d_{n+1}=3n^2+3n+1-2n^2-3n-1=n^2.$$
Thus,
$$\prod_{j=1}^{i-1} c_i \prod_{j=i+1}^{n+1} d_j=\frac{(i-1)!^2 (n+1)!(2n+1)!!}{i!(2i-1)!!}=A_n \frac{(i-1)!^2}{i!(2i-1)!!} =A_n \frac{i!^2 2^{i}}{i^2(2i)!},$$
where we use the fact that $i!(2i-1)!!=(2i)!/2^{i}.$ Therefore, we can express $B_n$ as $B_n=A_n \sum_{i=1}^{n+1}\frac{i!^2 2^{i}}{i^2(2i)!}$. Subsequently, utilizing Theorem~\ref{thm:convergents} and referencing \cite[Eq.3.78]{sherman2000summation}, we can deduce that
$$\lim_{n \to \infty} \frac{A_n}{B_n}=\frac{1}{\sum_{i=1}^{\infty}\frac{i!^2 2^{i}}{i^2(2i)!}}=\frac{8}{\pi^2}. \qedhere$$
\end{proof}

\begin{proof}[Proof of Theorem~\ref{thm:18pi2}]
The proof of Theorem~\ref{thm:18pi2} closely parallels the outlined approach, and we provide a concise summary here, omitting the detailed steps.

Choosing $d_1 = A_0/A_{-1} = b_0 = 2$, it becomes evident that $d_n = 2n(2n-1)$, resulting in
\[ A_n = \prod_{j=1}^{n+1}d_j = (2n+2)! \]

Utilizing the same variable $d_n$ to derive $B_n$, we note that $c_n = b_n - d_{n+1} = n^2$. This leads to the expression:
\[ B_n = \sum_{i=1}^{n+1}\prod_{j=1}^{i-1} c_i \prod_{j=i+1}^{n+1} d_j = A_n \sum_{i=1}^{n+1}\frac{(i-1)!^2 (2n+2)!}{(2i)!}. \]

By combining this result with the equation from \cite[Eq.3.75]{sherman2000summation} and leveraging Theorem~\ref{thm:convergents}, we obtain:
\[ \lim_{n \to \infty} \frac{A_n}{B_n} = \frac{1}{\sum_{i=1}^{\infty}\frac{(i-1)!^2 }{(2i)!}} = \frac{18}{\pi^2}. \qedhere \]
\end{proof}

\section{Discussion and Conclusion}

In this study, we introduced an alternative approach to solve second-order homogeneous equations with variable coefficients, departing from the methodology presented in \cite{mallik1997solution}. The core concept involves factoring the quadratic shift operator into linear operators and expressing the solution through terms featuring finite continued fractions. There exists potential for exploring analogous methods for higher-order equations.

While the proposed approach demonstrates simplicity and efficacy, particularly in applications to generalized continued fraction representations, it does not encompass all second-order linear difference equations. Consequently, an intriguing avenue for further research involves discerning the conditions under which this method is applicable. Specifically, it would be valuable to investigate a necessary and sufficient condition for the sequences $c_n$ and $d_n$ that satisfy \eqref{eqn:d+c} and \eqref{eqn:dc} to exist.

It is worth noting that solving second-order homogeneous equations with variable coefficients aids in converting infinite generalized continued fractions into series representations. However, it does not necessarily facilitate the determination of the limit. Consequently, unless the sum of the series representation is known, fully resolving the problem of representing a number with continued fractions remains an open challenge.

\section*{Disclosure statement}

The authors declare that they have no conflict of interest.

\section*{Funding}

Authors acknowledge the support by a grant from the Ministry of Science and Higher Education of the Republic of Kazakhstan within framework of the project AP19676669

\section*{Notes on contributor(s)}

All authors contributed to the work equally.

\end{document}